\documentclass{amsart}
\usepackage{amsthm}
\usepackage{amsmath}
\usepackage{amssymb}
\usepackage{amsfonts} 

\theoremstyle{plain}
\newtheorem{theorem}{Theorem}[section]
\newtheorem{corollary}[theorem]{Corollary}
\newtheorem{lemma}[theorem]{Lemma}

\newtheorem{example}[theorem]{Example}

\begin{document}

\title[A Profinite Module Invariant]{A Profinite Group Invariant for Hyperbolic Toral Automorphisms}
\author[Bakker]{Lennard F. Bakker}
\address{Department of Mathematics, Brigham Young University, Provo, Utah, USA}
\email{bakker@math.byu.edu}
\author[Martins Rodrigues]{Pedro Martins Rodrigues}
\address{Department of Mathematics, Instituto Superior T\'ecnico, Univ. Tec. Lisboa, Lisboa, Portugal}
\email{pmartins@math.ist.utl.pt}

\begin{abstract} For a hyperbolic toral automorphism, we construct a profinite completion of an isomorphic copy of the homoclinic group of its right action using isomorphic copies of the periodic data of its left action. The resulting profinite group has a natural module structure over a ring determined by the right action of the hyperbolic toral automorphism. This module is an invariant of conjugacy that provides means in which to characterize when two similar hyperbolic toral automorphisms are conjugate or not. In particular, this shows for two similar hyperbolic toral automorphisms with module isomorphic left action periodic data, that the homoclinic groups of their right actions play the key role in determining whether or not they are conjugate. This gives a complete set of dynamically significant invariants for the topological classification of hyperbolic toral automorphisms.
\end{abstract}

\maketitle

\section{Introduction} The topological classification of hyperbolic, irreducible, toral automorphisms brings together the subjects of dynamical systems, algebra, and algebraic number theory. It is well-known \cite{AP} that two $\mathbb{T}^{n}$ automorphisms induced by $A,B \in GL(n,\mathbb{Z})$ are topologically conjugate if and only if $A$ and $B$ are conjugate in within the group ${\rm GL}(n,{\mathbb Z})$, i.e., there is $C\in{\rm GL}(n,{\mathbb Z})$ such that $AC=CB$. Furthermore, by a well-known result of Latimer, MacDuffee and Taussky (see \cite{Ne,Ta, Wa}), this happens if and only if  $A$ and $B$ are associated to the same ideal class in the number ring determined by their common characteristic polynomial. In both these settings, algorithms have been developed that, in principle, determine when two automorphisms are conjugate \cite{Co,GS}. For instance, the case of $n=2$ uses classical results about continued fractions \cite{ATW, MS1}, while the case $n\geq 3$ uses recently developed geometric continued fractions \cite{Ko}. Yet there is not yet a clear understanding of the topological classification as would be given by a complete set of dynamically significant invariants. 

The periodic data for a hyperbolic, irreducible, toral automorphism is insufficient to characterize its conjugacy class, and yet provides a dynamically significant invariant of conjugacy (see \cite{MS2}). Let ${\rm Per}_{l}(A)$ denote the finite group of $l$-periodic points for the left action of $A$ in $\mathbb{T}^{n}=\mathbb{R}^{n}/\mathbb{Z}^{n}$, represented by columns vectors. It is possible for hyperbolic, irreducible, toral automorphisms $A$ and $B$ to have conjugate actions on the ${\rm Per}_{l}$ level, for every $l$, without $A$ and $B$ being conjugate while being similar, i.e., there is $C\in{\rm GL}(n,{\mathbb Q})$ such that $AC=CB$. More generally, suppose that $A$ and $B$ have the same irreducible characteristic polynomial $p(x)$, let $R=\mathbb{Z}[x]/(p(x))$, and set ${\rm Per}_{g}(A)={\rm ker}(g(A)) \subset \mathbb{T}^{n}$, where $g\in{\mathbb Z}[x]$ with $g(A)$ invertible. It is still possible that $A$ and $B$ are not conjugate on ${\mathbb T}^n$, and yet are strongly Bowen-Franks equivalent (also written strongly BF-equivalent), i.e., associated to every $g\in {\mathbb Z}[x]$ with $g(A)$ invertible, is an $R$-module isomorphism $\phi_g:{\rm Per}_{g}(A)\to{\rm Per}_{g}(B)$ that conjugates $A$ and $B$ on the level of periodic points determined by $g$. It is then natural to consider what additional dynamically significant invariants are needed to determine when two strongly BF-equivalent similar hyperbolic, irreducible, toral automorphisms are conjugate or not, i.e., when the level conjugacies would imply the existence of a global conjugacy.  

This paper explores this approach, from a dual point of view: instead of the direct limit of the $R$-modules ${\rm Per}_{g}(A)$, the main object is a profinite limit of the Pontryagin dual $R$-modules $\mathbb{Z}^{n}/\mathbb{Z}^{n}g(A)$, given by the right action of $g(A)$ on ${\mathbb Z}^n$ represented by row vectors. We review in Section 2 some basic definitions and facts about profinite groups. Then in Section 3, we present the construction of a profinite group $G_{A}$ associated to a hyperbolic $A\in{\rm GL}(n,{\mathbb Z})$. This $G_A$ is an $R$-module. It is a profinite completion of the right ${\mathbb Z}[A]$-module ${\mathbb Z}^n$ which is $R$-module isomorphic to the homoclinic group $H_A^\prime$ of the right action of hyperbolic toral automorphism induced by $A$. As an invariant of conjugacy, the $R$-module $H_A^\prime$ naturally embeds into $G_A$ and plays the key role in determining when two strongly BF-equivalent similar hyperbolic toral automorphisms are conjugate or not.

Obviously, strong BF-equivalence is a necessary condition for conjugacy of similar hyperbolic toral automorphisms. In Section 4, we describe conditions by which two similar hyperbolic toral automorphisms are strongly BF-equivalent in the dual sense, i.e., ${\mathbb Z}^n/{\mathbb Z}^n g(A)$ is $R$-module isomorphic to ${\mathbb Z}^n/{\mathbb Z}^n g(B)$ for every $g\in {\mathbb Z}[x]$ with $g(A)$ invertible. Then in Section 5 we show that strong BF-equivalence of similar hyperbolic $A,B\in{\rm GL}(n,{\mathbb Z})$ implies the existence of a topological $R$-module isomorphism $\Psi:G_A\to G_B$. In Section 6, we detail the exact manner by which embedded copies of $H_A^\prime$ and $H_B^\prime$ in $G_A$ and $G_B$ respectively, characterize when $A$ and $B$ are conjugate or not. Specifically, for strong BF-equivalent similar hyperbolic $A$ and $B$, it is how the image of the embedded copy of $H_A^\prime$ under any topological $R$-module isomorphism $\Psi:G_A\to G_B$ intersects the embedded copy of $H_B^\prime$ in $G_B$, that determines conjugacy or the lack thereof. This characterization of conjugacy applies when the characteristic polynomial of the similar $A$ and $B$ is irreducible or reducible. When the similar hyperbolic $A$ and $B$ have an irreducible characteristic polynomial, we show in Section 7 that the intersection of the image of the embedded copy of $H_A^\prime$ in $G_A$ under any topological $R$-module isomorphism $\Psi:G_A\to G_B$ with the embedded copy of $H_B^\prime$ in $G_B$ is either trivial or has finite index in the embedded copy of $H_B^\prime$ in $G_B$ 

The characterization of conjugacy for strongly BF-equivalent similar hyperbolic $A,B\in{\rm GL}(n,{\mathbb Z})$ show that the embedded copies of $H_A^\prime$ and $H_B^\prime$ in $G_A$ and $G_B$ are the additional dynamically significant invariants that are needed. It is known \cite{MS1.5} that the  homoclinic groups $H_A^\prime$ and $H_B^\prime$ are complete invariants of conjugacy when $A$ and $B$ are Pisot, i.e., have one real eigenvalue larger than $1$ in modulus with all the remaining eigenvalues smaller than $1$ in modulus. Our characterization of conjugacy extends the role of the homoclinic groups as classifying invariants from the Pisot case to the general case. This characterization of conjugacy also resolves the problem of classification of quasiperiodic flows of Koch type according to the equivalence relation of projective conjugacy (see \cite{Ba}), and gives a dynamical systems resolution to the ideal class problem in algebraic number theory.

From a computational point of view, our dynamical characterization of conjugacy for strongly BF-equivalent similar hyperbolic $A,B\in{\rm GL}(n,{\mathbb Z})$ is not completely satisfactory. We would like to better understand how to computationally detect when the conditions for characterization hold or are violated for the $R$-submodules $H_A^\prime$ and $H_B^\prime$ in $G_A$ and $G_B$ respectively. We are currently investigating some computational avenues for this detection that we hope to include in a subsequent paper.

\section{Basic Theory of Profinite Groups}\label{basic} We review the construction of and set the notation for profinite groups built from the left and right actions of an $A\in{\rm GL}(n,{\mathbb Z})$ on the integer lattice ${\mathbb Z}^n$. We list below many of the basic properties of these profinite groups. The statements and proofs of these basic properties are from \cite{RZ}, where the proofs appropriately adapt, when needed, from the group-theoretic setting to the $R$-module setting without difficulty. In the next section, we construct a particular profinite group with an $R$-module structure for a hyperbolic toral automorphism.  

The ring for the module structure on the profinite groups is ${\mathbb Z}[A]=\{q(A):q\in{\mathbb Z}[x]\}$ where ${\mathbb Z}[x]$ is the ring of polynomials with integer coefficients. The abelian group $G={\mathbb Z}^n$, represented by row vectors, is naturally a ${\mathbb Z}[A]$-module with the right action of $A$ on $m\in{\mathbb Z}^n$ defined by $m\to mA$.

We define a topology on $G$ as follows. Let ${\mathcal N}$ be a collection of ${\mathbb Z}[A]$-submodules of ${\mathbb Z}^n$ of finite index that is filtered from below, i.e., for $N,M\in {\mathcal N}$ there is $K\in{\mathcal N}$ such that $K\leq N\cap M$. The elements of ${\mathcal N}$, considered as a fundamental system of neighbourhoods of the identity element $0$ of $G$, endows $G$ with a profinite topology.

We make ${\mathcal N}$ into a directed poset with the binary relation $\preceq$ defined by $M\preceq N$ when $M\geq N$. The binary relation $\preceq$ satisfies $N\preceq N$ for all $N\in{\mathcal N}$, $N\preceq M$ and $M\preceq K$ imply $N\preceq K$ for $M,N,K\in{\mathcal N}$, $N\preceq M$ and $M\preceq N$ imply $N=M$, and, as a consequence of the filtered from below property of ${\mathcal N}$, for $N,M\in{\mathcal N}$ there is $K\in{\mathcal N}$ such that $N\preceq K$ and $M\preceq K$. 

We use ${\mathcal N}$ to form a surjective inverse system of finite ${\mathbb Z}[A]$-modules. For $N,M\in{\mathcal N}$ with $M\preceq N$, let $\varphi_{NM}$ denote the canonical ${\mathbb Z}[A]$-module epimorphism from $G/N$ to $G/M$ defined by $[m]_N\to [m]_M$ where $[m]_N=m+N$ and $[m]_M = m+M$. For each $N\in{\mathcal N}$, the action of $A$ on $G/N$ is the automorphism $A_N$ defined by
\[ A_N([m]_N) = [mA]_N.\]
This $A_N$ induces the module structure of ${\mathbb Z}[A]$ on $G/N$. For each $N\in{\mathcal N}$ endow $G/N$ with the discrete topology. This makes each $G/N$ a compact Hausdorff ${\mathbb Z}[A]$-module, each $A_N$ a homeomorphism, and each $\varphi_{NM}$ continuous epimorphism for $N\preceq M$. The collection $\{ G/N,\varphi_{NM},{\mathcal N}\}$ is a surjective inverse system, i.e., each $\varphi_{NM}$ is a continuous epimorphism, and $\varphi_{NK} = \varphi_{NM}\varphi_{MK}$ whenever $K\preceq M\preceq N$. 

A profinite completion of $G$ with respect to its profinite topology defined by ${\mathcal N}$ is the inverse limit of the surjective inverse system $\{G/N,\varphi_{NM},{\mathcal N}\}$. The inverse limit is the ${\mathbb Z}[A]$-submodule,
\[ G_{\mathcal N} \equiv \lim_{\longleftarrow} {}_{N\in{\mathcal N}} G/N \leq \prod_{N\in{\mathcal N}} G/N,\]
the set consisting of the tuples $\{ [m_N]_N\}_{N\in{\mathcal N}}$ for $m_N\in {\mathbb Z}^n$ satisfying $[m_M]_M=\varphi_{NM}([m_N]_N) = [m_N]_M$ whenever $M\preceq N$. With the topology on $G_{\mathcal N}$ being that induced by the product topology on $\Pi_{N\in{\mathcal N}}G/N$, the inverse limit $G_{\mathcal N}$ is a nonempty, compact, Hausdorff, totally disconnected ${\mathbb Z}[A]$-submodule with canonical continuous ${\mathbb Z}[A]$-module epimorphisms $\varphi_N:G_{\mathcal N} \to G/N$ defined by
\[ \varphi_N( \{ [m_M]_M \}_{M\in{\mathcal N}}) = [m_N]_N\]
that satisfy the compatibility conditions $\varphi_{NM}\varphi_N = \varphi_M$ for $M\preceq N$. The collection
\[ \{ {\rm ker}(\varphi_N):N\in{\mathcal N}\}\]
of open ${\mathbb Z}[A]$-submodules of $G_{\mathcal N}$ form a fundamental system of open neighbourhoods of $0$ in $G_{\mathcal N}$. The continous epimorphisms $A_N \varphi_N: G_{\mathcal N} \to G/N$, $N\in{\mathcal N}$, satisfy the compatibility conditions $\varphi_{NM}A_N\varphi_N = A_M\varphi_M$ whenever $M\preceq N$, and so induce a topological automorphism $\Gamma_A$ on $G_{\mathcal N}$ defined by
\[ \Gamma_A( \{[m_N]_N\}_{N\in{\mathcal N}}) = \{ [m_NA]_N\}_{N\in{\mathcal N}}.\]
This $\Gamma_A$ induces the module structure of ${\mathbb Z}[A]$ on $G_{\mathcal N}$.

A ${\mathbb Z}[A]$-module homomorphic copy of $G$ naturally sits inside $G_{\mathcal N}$. There is a canonical ${\mathbb Z}[A]$-module homomorphism $\iota_{\mathcal N}:G\to G_{\mathcal N}$ defined by
\[ \iota_{\mathcal N}(m) = \{ [m]_N\}_{N\in{\mathcal N}},\]
where $\iota_{\mathcal N}(G)$ is a dense ${\mathbb Z}[A]$-submodule of $G_{\mathcal N}$. This implies that $G_{\mathcal N}$ is finitely-generated, i.e.,
\[ G_{\mathcal N}=\overline{\langle \iota_{\mathcal N}(e_1),\dots,\iota_{\mathcal N}(e_n)\rangle},\]
which is the closure of group generated by the images under $\iota_{\mathcal N}$ of the standard basis of row vectors $e_1,\dots,e_n$ for $G$. The ${\mathbb Z}[A]$-homomorphism $\iota_{\mathcal N}$ is a ${\mathbb Z}[A]$-monomorphism if and only if 
\[ \bigcap_{N\in{\mathcal N}} N = \{0\}.\]
When $\iota_{\mathcal N}$ is a ${\mathbb Z}[A]$-module monomorphism, the ${\mathbb Z}[A]$-module $G_{\mathcal N}$ is uncountable.

The ${\mathbb Z}[A]$-module $G_{\mathcal N}$ is realized ${\mathbb Z}[A]$-module isomorphically  through all of its open ${\mathbb Z}[A]$-submodules. A ${\mathbb Z}[A]$-submodule $U$ of $G_{\mathcal N}$ is open in $G_{\mathcal N}$ if and only if $U$ is closed of finite index in $G_{\mathcal N}$. Because $G_{\mathcal N}$ is finitely-generated, there is for each $l\in{\mathbb N}$ only finitely many open ${\mathbb Z}[A]$-submodules of $G_{\mathcal N}$ of index $l$. The collection ${\mathcal U}$ of all of the open ${\mathbb Z}[A]$-submodules of $G_{\mathcal N}$ is filtered from below and forms a fundamental system of neighbourhoods of $0$ in $G_{\mathcal N}$ that satisfy
\[ \bigcap _{U\in{\mathcal U}} U = \{0\}.\]
The collection ${\mathcal U}$ includes the collection $\{ {\rm ker}(\varphi_N):N\in{\mathcal N}\}$. We make ${\mathcal U}$ into a directed poset by defining $V\preceq U$ when $U\leq V$. This gives a surjective inverse system $\{G_{\mathcal N}/U,\varphi_{UV},{\mathcal U}\}$ where $\varphi_{UV}:G_{\mathcal N}/U \to G_{\mathcal N}/V$ are the continuous canonical ${\mathbb Z}[A]$-module epimorphisms. There is a topological ${\mathbb Z}[A]$-module isomorphism $\psi$ from $G_{\mathcal N}$ to the inverse limit of $\{G_{\mathcal N}/U,\varphi_{UV},{\mathcal U}\}$, i.e., to
\[ G_{\mathcal U}=\lim_{\longleftarrow}{}_{U\in{\mathcal U}} G_{\mathcal N}/U\]
with the continuous canonical epimorphisms $\varphi_U:G_{\mathcal U}\to G_{\mathcal N}/U$. The map $\psi$ is induced by the continuous canonical epimorphisms $\psi_U:G_{\mathcal N}\to G_{\mathcal N}/U$ that are compatible, i.e., $\varphi_{UV}\psi_U=\psi_V$ when $V\preceq U$, such that $\varphi_U\psi=\psi_U$ holds for all $U\in{\mathcal U}$.

A subcollection of open ${\mathbb Z}[A]$-submodules of $G_{\mathcal N}$ may sometimes realize $G_{\mathcal N}$ as well. A cofinal subsystem of $\{ G_{\mathcal N}/U,\varphi_{UV},{\mathcal U}\}$ is a surjective inverse system $\{G_{\mathcal N}/K,\varphi_{KL},{\mathcal K}\}$ where  ${\mathcal K}$ is a subset of ${\mathcal U}$ such that ${\mathcal K}$ is a directed poset with respect to the binary relation $\preceq$ on ${\mathcal U}$, and such that for each $U\in{\mathcal U}$ there is $K\in{\mathcal K}$ with $U\preceq K$. When $\{G_{\mathcal N}/K,\varphi_{KL},{\mathcal K}\}$ is a cofinal subsystem of $\{ G_{\mathcal N}/U,\varphi_{UV},{\mathcal U}\}$, there is a topological ${\mathbb Z}[A]$-module isomorphism $\phi:G_{\mathcal K}\to G_{\mathcal U}$  where $G_{\mathcal K}$ is the inverse limit of the cofinal subsystem with its continuous canonical epimorphisms $\varphi_K: G_{\mathcal N} \to G_{\mathcal N}/K$. The map $\phi$ is induced by the continuous canonical epimorphisms $\bar\varphi_U: G_{\mathcal K}\to G_{\mathcal N}/U$ defined independently of $K$ by $\bar\varphi = \varphi_{KU} \varphi_K$ when $U\preceq K$, which are compatible, i.e., $\varphi_{UV}\bar\varphi_U=\bar\varphi_V$ when $V\preceq U$, such that $\varphi_U\phi=\bar\varphi_U$ holds for all $U\in {\mathcal U}$.

\section{A ${\mathbb Z}[A]$-Module For Hyperbolic $A$} We use the Pontryagin dual of periodic data for the left action of a hyperbolic toral automorphism to construct a profinite group associated to it. Let $A\in{\rm GL}(n,{\mathbb  Z})$ be hyperbolic, and let $T_A$ be the hyperbolic toral automorphism that the left action of $A$ induces on ${\mathbb T}^n$, i.e., $T_A\pi = \pi A$ where $\pi:{\mathbb R}^n\to {\mathbb T}^n$ is the canonical covering epimorphism with ${\mathbb R}^n$ and ${\mathbb T}^n$ represented by column vectors. Hyperbolicity of $A$ implies that ${\rm det}(A^r-I)\ne 0$ for all $r\in{\mathbb Z}$. Let $T_A^\prime$ be the hyperbolic toral automorphism that the right action $A$ on ${\mathbb Z}^n$, represented by row vectors, induces on ${\mathbb T}^n$, represented by row vectors. Let 
\[ H_A^\prime =\{ x\in {\mathbb T}^n: x(T_A^\prime)^k\to 0{\rm\ as\ }\vert k\vert\to\infty\}\]
 denote the homoclinic group for $T^\prime_A$, i.e., the homoclinic group for the right action of $T_A$. Then there is a ${\mathbb Z}[A]$-module isomorphism $\theta_A:{\mathbb Z}^n\to H_A^\prime$ from the right ${\mathbb Z}[A]$-module ${\mathbb Z}^n$ to the right ${\mathbb Z}[A]$-module $H_A^\prime$ (see \cite{LS}). For each $k\in{\mathbb Z}$, we also use the right action of $A$ on ${\mathbb Z}^n$ to define the abelian groups
\[ N_{k,A} = {\mathbb Z}^n(A^{k!}-I) {\rm \ and\ } G_{k,A} = {\mathbb Z}^n/N_{k,A}.\]
The finite abelian groups $G_{k,A}$, having  elements of the form $[m]_{k,A}=m+N_{k,A}$, are generalized Bowen-Franks groups  which are isomorphic to the Pontryagin duals of ${\rm Per}_{k!}(T_A)$, the group of periodic points of period $k!$ for the left action $T_A$ on ${\mathbb T}^n$. Obviously, the collection ${\mathcal N}_A=\{ N_{k,A}:k\in{\mathbb N}\}$ consists of ${\mathbb Z}[A]$-submodules of ${\mathbb Z}^n$ of finite index. 

\begin{lemma}\label{NA} For any $A\in{\rm GL}(n,{\mathbb Z})$, the collection ${\mathcal N}_A$ satisfies $N_{k,A}\geq N_{k+1,A}$ for all $k\in{\mathbb N}$ and is filtered from below. If, in addition, $A$ is hyperbolic, then ${\mathcal N}_A$ also satisfies
\[ \bigcap_{k\in{\mathbb N}} N_{k,A} = \{0\}.\]
\end{lemma}

\begin{proof} First we show that $N_{k,A}\geq N_{k+1,A}$. For each $k\in{\mathbb N}$, there is the factorization
\[ A^{(k+1)!} - I = (A^{(k+1)!-k!} + A^{(k+1)!-2k!} + \cdot\cdot\cdot + A^{k!} + I)(A^{k!}-I),\]
where within the first set of braces on the right-hand side, the power of $A$ eventually becomes $(k+1)!-kk!=k!$. The map
\[ \xi \to \xi (A^{(k+1)!-k!} + A^{(k+1)!-2k!} + \cdot\cdot\cdot + A^{k!} + I) \]
is an endomorphism of ${\mathbb Z}^n$. Hence
\begin{align*} {\mathbb Z}^n(A^{(k+1)!}-I) & = {\mathbb Z}^n(A^{(k+1)!-k!} + A^{(k+1)!-2k!} + \cdot\cdot\cdot + A^{k!} + I)(A^{k!}-I) \\
&\leq {\mathbb Z}^n(A^{k!}-I),
\end{align*}
that is $N_{k,A}\geq N_{k+1,A}$.

Next we show that ${\mathcal N}_A$ is filtered from below. For $k_3=k_1+k_2$ with $k_1,k_2\in{\mathbb N}$, there is the factorization
\[ A^{k_3!}-I = (A^{k_3!-k_1!} + A^{k_3!-2k_1!} + \cdot\cdot\cdot + A^{k_1!} + I)(A^{k_1!}-I).\]
In the first set of braces on the right-hand side, the power of the $A$ eventually becomes
\[ k_3! - [k_3(k_3-1)\cdot\cdot\cdot(k_1+1)-1]k_1!=k_1!.\]
Since the term in the first set of braces on the right side of the factorization is an endomorphism of ${\mathbb Z}^n$, it follows that $N_{k_1,A}\geq N_{k_3,A}$. Similarly, $N_{k_2,A}\geq N_{k_3,A}$, and so $N_{k_3,A}\leq N_{k_1,A}\cap N_{k_2,A}$. Hence ${\mathcal N}_A$ is filtered from below.

Finally we show that the intersection of all the $N_{k,A}$ is $\{0\}$ by contradiction. Suppose there is $m\in{\mathbb Z}^n\setminus\{0\}$ such that $m\in\cap_{k\in{\mathbb N}} N_{k,A}$. Then for each $k\in{\mathbb N}$, there is $\xi_k\in{\mathbb Z}^n\setminus\{0\}$ such that $m = \xi_k(A^{k!}-I)$. This means that $A^{k!}$ sends an element of ${\mathbb Z}^n$ to another element of ${\mathbb Z}^n$ that is a distance of $\Vert m\Vert$ away from $\xi_k$. (Here $\Vert\cdot\Vert$ is the usually Euclidean norm on ${\mathbb R}^n$.) Let $E^s(A)$ and $E^u(A)$ be the stable and unstable spaces of $A$ acting on ${\mathbb R}^n$. The only element of ${\mathbb Z}^n$ that belongs to $E^s(A)\cup E^u(A)$ is $0$ because no lattice point can iterate under $A$ or $A^{-1}$ to $0$. The  hyperbolicity of $A$ implies that $\Vert \xi A^{k!}\Vert$, for $\xi\in{\mathbb Z}^n\setminus\{0\}$, goes to $\infty$ with exponential speed as $k\to\infty$ (see Proposition 1.2.8 on p.\,28 in \cite{KH}). This means that for all sufficiently large $k$, each $A^{k!}$ sends the points in ${\mathbb Z}^n$ a distance of $\Vert m\Vert$ away from $\xi_k$, to a distance much larger than $\Vert m\Vert$ away from $\xi_k$. This is a contradiction.
\end{proof}

We associate to a hyperbolic $A\in{\rm GL}(n,{\mathbb Z})$ a ${\mathbb Z}[A]$-module $G_A$ as follows, where, to set some notation in this context, we state some of the basic properties from Section \ref{basic} without explicit reference. By Lemma \ref{NA}, the binary relation $\preceq$ on the filtered from below collection ${\mathcal N}_A$, indexed by ${\mathbb N}$, simplifies to $l\leq k$ for $l,k\in{\mathbb N}$ implies $N_{l,A}\geq N_{k,A}$. We denote the elements of $G_{k,A}$ by $[m]_{k,A}=m+N_{k,A}$ for $m\in{\mathbb Z}^n$. The ${\mathbb Z}[A]$-module structure on $G_{k,A}$ is induced by the automorphism $A_k$ of $G_{k,A}$ defined by $A_k([m]_{k,A}) = [mA]_{k,A}$. For $l\leq k$, there are the continuous canonical ${\mathbb Z}[A]$-module epimorphisms $\varphi_{k,l,A}:G_{k,A}\to G_{l,A}$ defined by $\varphi_{k,l,A}([m]_{k,A}) = [m]_{l,A}$. The inverse limit of the surjective inverse system $\{ G_{k,A},\varphi_{k,l,A},{\mathbb N}\}$ is the compact, Hausdorff, totally disconnected, ${\mathbb Z}[A]$-module
\[ G_A \equiv \lim_{\longleftarrow}{}_{k\in{\mathbb N}} G_{k,A} \leq  \prod_{k\in{\mathbb N}} G_{k,A}\]
consisting of the tuples $\{[m_k]_{k,A}\}_{k\in{\mathbb N}}$ for $m_k\in {\mathbb Z}$ satisfying $[m_l]_{l,A}=\varphi_{kl}([m_k]_{k,A})=[m_k]_{l,A}$ when $l\leq k$, together with the continuous canonical ${\mathbb Z}[A]$-module epimorphisms $\varphi_{k,A}:G_A\to G_{k,A}$ defined by
\[ \varphi_{k,A}( \{[m_k]_{k,A}\}_{k\in{\mathbb N}}) = [m_k]_{k,A}.\]
The collection $\{ {\rm ker}(\varphi_{k,A}):k\in{\mathbb N}\}$ form a fundamental system of open ${\mathbb Z}[A]$-submodules of $0$ in $G_A$. The ${\mathbb Z}[A]$-module structure on $G_A$ is induced by the topological automorphism $\Gamma_A$ defined by $\Gamma_A( \{[m_k]_{k,A}\}_{k\in{\mathbb N}}) = \{A_k[m_k]_{k,A}\}_{k\in{\mathbb N}}$. By Lemma \ref{NA}, the canonical ${\mathbb Z}[A]$-module homomorphism $\iota_A:{\mathbb Z}^n\to G_A$ defined by $\iota_A(m) = \{[m]_{k,A}\}_{k\in{\mathbb N}}$ is a monomorphism, and so $G_A$ is uncountable.

We could have constructed $G_A$ from the filtered from below directed poset ${\mathcal N}$ of all the finite ${\mathbb Z}[A]$-modules $BF_g(A)$ for $g\in{\mathbb Z}[x]$ with $g(A)$ invertible. However, because ${\mathcal N}_A$ is cofinal in ${\mathcal N}$, we would have constructed a ${\mathbb Z}[A]$-module from ${\mathcal N}$ that would have been ${\mathbb Z}[A]$-module isomorphic to $G_A$. We prefer the cofinal subcollection because of its linear ordering.

For another hyperbolic $B\in{\rm GL}(n,{\mathbb Z})$ that is similar to $A$, there is a ring $R$ such that $G_{k,A}$, $G_{k,B}$, $G_A$, $G_B$, etc., are all $R$-modules. Similarity of $A$ and $B$ implies that ${\mathbb Z}[A]$ and ${\mathbb Z}[B]$ are isomorphic as rings. We denote by $R$ a ring that is isomorphic to ${\mathbb Z}[A]$ and ${\mathbb Z}[B]$, and understand now that the $R$-module structure on $G_A$ and $G_B$ is given by ${\mathbb Z}[A]$ and ${\mathbb Z}[B]$ respectively. When $A$ and $B$ have the same irreducible characteristic polynomial $p(x)$, we take $R={\mathbb Z}[x]/\big(p(x)\big)$, where $\big(p(x)\big)=p(x){\mathbb Z}[x]$ is the principal ideal in ${\mathbb Z}[x]$ generated by $p(x)$. When the characteristic polynomial $p(x)$ of $A$ and $B$ is not irreducible, we take $R=\big( {\mathbb Z}[x]/\big(p(x)\big)\big)/{\rm ker}(\theta_D)$ where, for $D=A$ or $D=B$, the map $\theta_D:{\mathbb Z}[x]/\big(p(x)\big)\to {\mathbb Z}[D]$ is the ring endomorphism defined by $\vartheta_D\big( q(x)+(p(x))\big) = q(D)$; it does matter whether $D=A$ or $D=B$ because the similarity of $A$ and $B$ implies that ${\rm ker}(\theta_A)={\rm ker}(\theta_B)$.

\section{Sufficient Conditions for Strong BF-Equivalence} For similar hyperbolic $A,B\in{\rm GL}(n,{\mathbb Z})$, the finite $R$-modules $G_{k,A},G_{k,B}$ for $k\in{\mathbb N}$ are invariants of the conjugacy classes of $A$ and $B$ respectively in terms of their $R$-module structure. More generally, for $g\in{\mathbb Z}[x]$ with $g(A)$ (and hence $g(B)$) invertible, the finite $R$-modules ${\mathbb Z}^n/{\mathbb Z}^n g(A)$ and ${\mathbb Z}^n/{\mathbb Z}^n g(B)$ are invariants of the conjugacy classes of $A$ and $B$ in terms of their $R$-module structures. The generalized Bowen-Franks groups are ${\rm BF}_g(A)={\mathbb Z}^n/{\mathbb Z}^n g(A)$ for $g\in{\mathbb Z}[x]$ with $g(A)$ invertible. We note that $G_{k,A}={\rm BF}_{g_k}(A)$ where $g_k(x)=x^{k!}-1$, $k\in{\mathbb N}$. For two $R$-modules $G_1$ and $G_2$, we write $G_1\cong_R G_2$ when $G_1$ and $G_2$ are $R$-module isomorphic. Similar hyperbolic $A,B\in{\rm GL}(n,{\mathbb Z})$ are said to be strongly BF-equivalent when ${\rm BF}_g(A)\cong_R{\rm BF}_g(B)$ for all all $g(x)\in{\mathbb Z}[x]$ with $g(A)$ invertible. When ${\rm BF}_g(A)$ and ${\rm BF}_ g(B)$ are not $R$-module isomorphic for some $g\in{\mathbb Z}[x]$ with $g(A)$ invertible, then the similar $A$ and $B$ fail to be conjugate. So a necessary condition for similar hyperbolic $A$ and $B$ to be conjugate is that $A$ and $B$ are strongly BF-equivalent.

\begin{example}\label{ex1}{\rm The hyperbolic matrices
\[ A = \begin{bmatrix} 0 & 1 & 0 \\ 1 & 0 & 4 \\ 6 & -2 & 23\end{bmatrix}, \ \ B=\begin{bmatrix} 0 & 1 & 12 \\ 1 & 0 & -4 \\ 0 & 2 & 23\end{bmatrix}\]
have the same irreducbile characteristic polynomial of $p(x) = x^3-23x^2+7x-1$. Using the Smith Normal Form, we compute that ${\rm BF}_{x+1}(A)={\mathbb Z}_4\oplus{\mathbb Z}_{8}$ while ${\rm BF}_{x+1}(B) = {\mathbb Z}_2\oplus{\mathbb Z}_{16}$. Thus $A$ and $B$ are not strongly BF-equivalent, and hence are not conjugate.
}\end{example}

The conditions by which two hyperbolic toral automorphisms with a common irreducible characteristic polynomial $p(x)\in{\mathbb Z}[x]$ are strongly BF-equivalent are based in algebraic number theory, where the connection is given by the Latimer-MacDuffee-Taussky Theorem. The field of fractions of $R={\mathbb Z}[x]/(p(x))$ is the algebraic number field ${\mathbb K}={\mathbb Q}[x]/(p(x))$. For a fixed root $\beta$ of $p(x)$, we make the identifications $R={\mathbb Z}[\beta]$ and ${\mathbb K}={\mathbb Q}(\beta)$. For hyperbolic $A\in{\rm GL}(n,{\mathbb Z})$ with characteristic polynomial $p(x)$, choose a row eigenvector $v$ of $A$ for $\beta$, i.e., $vA =\beta v$, such that the components of $v$ form a ${\mathbb Z}$-basis for a fractional ideal $I$ of ${\mathbb Z}[\beta]$, i.e., a finitely generated ${\mathbb Z}[\beta]$-module contained in ${\mathbb K}$. Another fractional idea $J$ of ${\mathbb Z}[\beta]$ is said to be equivalent to $I$ if there is a nonzero $z\in{\mathbb K}$ such $J=zI$. The Latimer-MacDuffee-Taussky Theorem asserts that there is a one-to-one correspondence between the conjugacy classes of matrices in ${\rm GL}(n,{\mathbb Z})$ with characteristic polynomial $p(x)$ and the equivalence classes of the fractional ideals of ${\mathbb Z}[\beta]$. With this one-to-one correspondence, we have that $G_{k,A}\cong_R I/g_k(\beta)I$, where $g_k(x)=x^{k!}-1$, $k\in{\mathbb N}$. The ring of coefficients of $I$ is
\[{\mathfrak O}(I) = \{z\in {\mathbb K}: zI\subset I\};\]
it is contained by the ring of integers ${\mathfrak o}_{\mathbb K}$ of ${\mathbb K}$, and it contains ${\mathbb Z}[\beta]$. A nonzero fractional ideal $I$ of ${\mathbb Z}[\beta]$ is said to be invertible in ${\mathfrak O}(I)$ if the fractional ideal
\[ I^{-1}=\{z\in {\mathbb K}: zI\subset {\mathfrak O}(I)\}\] 
satisfies $II^{-1}={\mathfrak O}(I)$. For $B\in{\rm GL}(n,{\mathbb Z})$ with characteristic polynomial $p(x)$, choose a row eigenvector $w$ for $B$ and $\beta$, i.e., $wB=\beta w$, and associate to $B$ the fractional ideal $J$ of ${\mathbb Z}[\beta]$ generated by $w$. The ring of coefficients of $J$ is an invariant of conjugacy class of $B$ in that if $A$ and $B$ are conjugate, then $I$ and $J$ have the same ring of coefficients. We may choose the associated fractional ideals $I$ and $J$ so that $I\subset J\subset {\mathbb Z}[\beta]$. Suppose that $I$ and $J$ have the same ring of coefficients ${\mathfrak O}$. We say that $I$ and $J$ are weakly equivalent if there are (finitely generated) ${\mathfrak O}$-modules $X$ and $Y$ of ${\mathbb K}$ such that $IX=J$ and $JY=I$. When $I$ and $J$ are weakly equivalent, we have that $X=\{ z\in {\mathbb K}: zI\subset J\}$ and $Y=\{ z\in{\mathbb K}:zJ\subset I\}$, and $XY={\mathfrak O}$, where the latter condition implies that $X$ and $Y$ are invertible fractional ideals of ${\mathfrak O}$. Equivalent fractional ideals in ${\mathbb Z}[\beta]$ are always weakly equivalent. It is known \cite{MS2} that if $A$ and $B$ are strongly BF-equivalent, then ${\mathfrak O}(I)={\mathfrak O}(J)$.

\begin{theorem}\label{SBFsufficient} For hyperbolic $A,B\in{\rm GL}(n,{\mathbb Z})$ with the same irreducible characteristic polynomial $p(x)$, suppose that the associated fractional ideals $I$ and $J$ of ${\mathbb Z}[\beta]$ have the same ring of coefficients ${\mathfrak O}$. If $I$ and $J$ are weakly equivalent, then $A$ and $B$ are strongly BF-equivalent.
\end{theorem}

\begin{proof} Set $\alpha=g(\beta)\in{\mathbb Z}[\beta]$ for $g\in{\mathbb Z}[x]$ with $g(A)$ invertible. Because $X$ and $Y$ are invertible fractional ideals of ${\mathfrak O}$, there are constants $\gamma,a,b\in{\mathfrak O}$ such that
\[ X = \alpha {\mathfrak O} + \gamma{\mathfrak O}, \ Y = a{\mathfrak O}+b{\mathfrak O}, \ a\alpha+b\gamma=1.\]
Consider the map from $I$ to $J$ given by multiplication by $\gamma$, i.e., $x\in I$ gets maps to $\gamma x\in J$. It obviously maps $\alpha I$ into $\alpha J$. On the other hand, if $\gamma x\in \alpha J$, say $\gamma x=\alpha z$ for $z\in J$, then
\[ x = a\alpha x+b\gamma x = \alpha(ax+bz).\]
Since $I\subset J\subset {\mathbb Z}[\beta]\subset {\mathfrak O}$, it follows that $x\in \alpha I$. So $\gamma$ induces an isomorphism
\[ \gamma_*:I/\alpha I \to J/\alpha J.\]
Using ${\mathbb Z}$-bases for $I$ and $J$ we get an invertible integer-entry matrix $X_g$ defined by $\gamma v =w X_g$ that satisfies \[ X_g A = BX_g.\]
This semi-conjugacy induces an $R$-module isomorphism between the groups ${\rm BF}_g(A)$ and ${\rm BF}_g(B)$. Since $g\in{\mathbb Z}[x]$ with $g(A)$ invertible is arbitrary, we have strong BF-equivalence of $A$ with $B$.
\end{proof}

Theorem \ref{SBFsufficient} extends a result from \cite{MS2}, which states that for hyperbolic $A,B\in{\rm GL}(n,{\mathbb Z})$ with the same irreducible characteristic polynomial $p(x)$, whose associated fractional ideals $I$ and $J$ of ${\mathbb Z}[\beta]$ satisfy ${\mathfrak O}(I)={\mathfrak O}(J)$ with $I$ and $J$ invertible in ${\mathfrak O}(I)$, the matrices $A$ and $B$ are strongly BF-equivalent. Invertible fractional ideals of ${\mathbb Z}[\beta]$ are always weakly equivalent. One way to guarantee the invertibility of nonzero fractional ideals of ${\mathbb Z}[\beta]$ is when the discriminant of $p(x)$ is square-free; then ${\mathbb Z}[\beta]={\mathfrak o}_{\mathbb K}$ (see \cite{Co}), a Dedekind domain in which every nonzero fractional ideal is invertible (see \cite{SD}). 

\begin{example}\label{ex2}{\rm  Consider the irreducible polynomial $p(x)=x^3-2x^2-8x-1$ whose discriminant of $1957$ is square-free (see Table B4 in \cite{Co}). The two matrices
\[ A = \begin{bmatrix}0 & 1 & 0 \\ 0 & 0 & 1 \\ 1 & 8 & 2 \end{bmatrix}, \ \ B=\begin{bmatrix}-1  & 2 & 0 \\ -1 & 1 & 1 \\ -5 & 9 & 2\end{bmatrix} \]
have $p(x)$ as their characteristic polynomials, where $A$ is the the companion matrix for $p(x)$. Let $\beta$ be a root of $p(x)$, set ${\mathbb K}={\mathbb Q}(\beta)$, and let $I$ and $J$ be fractional ideals in ${\mathbb Z}[\beta]$ associated to $A$ and $B$. Then ${\mathbb Z}[\beta]={\mathfrak o}_{\mathbb K}$ because the discriminant of $p(x)$ is square-free, and so ${\mathfrak O}(I)={\mathfrak O}(J)$ and $I$ and $J$ are invertible in ${\mathfrak o}_{\mathbb K}$. By Theorem \ref{SBFsufficient} and the comments that followed its proof, the matrices $A$ and $B$ are strongly BF-equivalent. But, $I$ and $J$ are inequivalent fractional ideals (see \cite{KKS}), and so $A$ and $B$ are not conjugate.
}\end{example}

\section{Strongly BF-Equivalent Hyperbolic Toral Automorphisms} Unfortunately, for similar hyperbolic $A,B\in{\rm GL}(n,{\mathbb Z})$, strong BF-equivalence of $A$ and $B$ does not imply the conjugacy of $A$ and $B$ (see \cite{MS2}). As illustrated in Example \ref{ex2}, an algebraic obstruction for this is the existence of inequivalent invertible ideals with the same ring of coefficients. However, strong BF-equivalence does relate the profinite $R$-modules $G_A$ and $G_B$.

\begin{theorem}\label{SBFimpliesPsi} Suppose hyperbolic $A,B\in{\rm GL}(n,{\mathbb Z})$ are similar. If $A$ and $B$ are strongly BF-equivalent, then there exists a topological $R$-module isomorphism $\Psi:G_A\to G_B$.
\end{theorem}

\begin{proof} Under the hypotheses on $A$ and $B$, we prove $G_A\cong_R G_B$ through the open $R$-submodules of $G_A$ and $G_B$. Let ${\mathcal U}_A$ be the collection of all open $R$-submodules of $G_A$, and ${\mathcal U}_B$ the collection of all open $R$-submodules of $G_B$.

The first claim is that $G_A$ and $G_B$ have the same finite quotients through their respective open $R$-submodules:
\[ \{ G_A/U:U\in{\mathcal U}_A\} = \{ G_B/V:V\in {\mathcal U}_B\},\]
where we understand equality in terms of $R$-module isomorphisms. Let $H=G_A/U$ for $U\in {\mathcal U}_A$. Since $\{ {\rm ker}(\varphi_{k,A}):k\in{\mathbb N} \}$ is a fundamental set of open $R$-submodule neighbourhoods of $0$, there is a $k\in{\mathbb N}$ such that ${\rm ker}(\varphi_{k,A})$ is an $R$-submodule of $U$. Thus $U/{\rm ker}(\varphi_{k,A})$ is an $R$-submodule of $G_A/{\rm ker}(\varphi_{k,A})$ and $(G_A/{\rm ker}(\varphi_{k,A}))/(U/{\rm ker}(\varphi_{k,A}))\cong_R G_A/U=H$, where $G_A/{\rm ker}(\varphi_{k,A})\cong_R G_{k,A}$. By the assumption of strong BF-equivalence of $A$ and $B$, we have that $G_{k,A}\cong_R G_{k,B}\cong_R G_B/{\rm ker}(\varphi_{k,B})$. Then there is an open $R$-submodule $V$ of $G_B$ containing ${\rm ker}(\varphi_{k,B})$ such that $U/{\rm ker}(\varphi_{k,A})\cong_R V/{\rm ker}(\varphi_{k,B})$. This implies that $H=G_A/U\cong_R G_B/V$ for $V\in {\mathcal U}_B$. Reversing the roles of $A$ and $B$ in this argument gives the other inclusion, and thus the first claim.

From this point on, we adapt the argument from (\cite{RZ}, Theorem 3.2.7 on pp.\,88-89). For each $k\in{\mathbb N}$ set
\[ U_k = \cap\{ U\in {\mathcal U}_A:[G_A,U]\leq k\}, \ \ V_k = \cap \{ V\in {\mathcal U}_B:[G_B:V]\leq k\}.\]
Since $G_A$ and $G_B$ are both finitely generated, there are for each $k\in{\mathbb N}$ only finitely many $U\in {\mathcal U}_A$ and $V\in {\mathcal U}_B$ such that $[G_A:U]\leq k$ and $[G_B:V]\leq k$. Thus $U_k\in {\mathcal U}_A$ and $V_k\in {\mathcal U}_B$ for all $k\in{\mathbb N}$. By the first claim, for each $k\in{\mathbb N}$ there is a $K\in {\mathcal U}_A$ such that $G_A/K\cong_R G_B/V_k$.

The second claim is that this $K$ is an $R$-submodule of $U_k$. We show this by setting up an association between certain open $R$-submodules of $G_A$ and $G_B$. Take any $V\in {\mathcal U}_B$ with $[G_B:V]\leq k$. Since $V_k$ is an open $R$-submodule of $V$, we have that $V/V_k$ is an $R$-submodule of $G_B/V_k$ with
\[ (G_B/V_k)/(V/V_k)\cong_R G_B/V.\]
Since $G_A/K\cong_R G_B/V_k$, there is by the first claim, an $U\in {\mathcal U}_A$ with $K$ an $R$-submodule of $U$ such that
\[ (G_B/V_k)/(V/V_k)\cong_R (G_A/K)/(U/K)\cong_R G_A/U.\]
This implies that  $G_A/U\cong_R G_B/V$, and so $[G_A:U]=[G_B:V]\leq k$. Thus associated to every $V\in{\mathcal U}_B$ with $[G_B:V]\leq k$, there is an $U\in{\mathcal U}_A$ that has $K$ as an $R$-submodule and satisfies $[G_A:U]\leq k$, i.e., $U_k$ is an $R$-submodule of $U$. On the other hand, for any $U\in{\mathcal U}_A$ with $[G_A:U]\leq k$ and $K$ an $R$-submodule of $U$, we have $U/K$ is an $R$-submodule of $G_A/K$ with
\[ (G_A/K)/(U/K)\cong_R G_A/U.\] 
Since $G_A/K\cong_R G_B/V_k$, there is by the first claim a $V\in {\mathcal U}_B$ with $V_k$ an $R$-submodule of $V$ such that
\[ (G_A/K)/(U/K)\cong_R (G_2/V_k)/(V/V_k)\cong_R G_B/V.\]
This implies that $G_A/U\cong_R G_B/V$, and so $[G_B:V]=[G_A:U]\leq k$. Thus associated to each $U\in {\mathcal U}_A$ with $[G_A:U]\leq k$ and $K$ an $R$-submodule of $U$, is a $V\in{\mathcal U}_B$ with $[G_B:V]\leq k$. These two arguments account for all $U\in{\mathcal U}_A$ with $K$ an $R$-submodule of $U$ such that  $[G_A:U]\leq k$. Thus we obtain that $K$ is an $R$-submodule of $U_k$, giving the second claim.

The third claim is that $G_A/U_k\cong_R G_B/V_k$ for all $k\in{\mathbb N}$. Since $K$ is an $R$-submodule of $U_k$ by the second claim, it follows that $G_A/U_k$ is an $R$-submodule of $G_A/K$. Since $G_A/K\cong_R G_B/V_k$, we have that $\vert G_A/U_n\vert \leq \vert G_B/V_k\vert$ for all $k\in{\mathbb N}$. Reversing the roles gives by a similar argument that $\vert G_A/U_k\vert \geq \vert G_B/V_k\vert$, this giving $\vert G_A/U_k\vert=\vert G_B/V_k\vert $ for all $k\in{\mathbb N}$. Since $G_A/U_k$ is $R$-module isomorphically an $R$-submodule of $G_B/V_k$, the finiteness of $G_A/U_k$ and $G_B/V_k$ implies that $G_A/U_k\cong_R G_B/V_k$ for all $k\in{\mathbb N}$. This gives the third claim.

For each $k\in{\mathbb N}$, let $Z_k$ be the set of all $R$-module isomorphisms from $G_A/U_k$ to $G_B/V_k$. By the third claim, $Z_k\ne\emptyset$ for all $k\in{\mathbb N}$. Let $\sigma_k$ denote an element of $Z_k$.

The fourth claim is that each $R$-module isomorphism $\sigma_{k+1}\in Z_{k+1}$ induces an $R$-module isomorphism $\sigma_k\in Z_k$. Since $U_{k+1}$ is an $R$-submodule of $U_k$ and $V_{k+1}$ is an $R$-submodule of $V_k$, we have that $U_k/U_{k+1}$ is an $R$-submodule of $G_A/U_{k+1}$ and $V_k/V_{k+1}$ is an $R$-submodule of $G_B/V_{k+1}$. If $\sigma_{k+1}(U_k/U_{k+1}) = V_k/V_{k+1}$, then
\begin{align*} \sigma_{k+1}(G_A/U_k) 
& = \sigma_{k+1}\big( (G_A/U_{k+1})/(U_k/U_{k_1})\big) \\
& \cong_R \sigma_{k+1}(G_A/U_{k+1})/\sigma_{k+1}(U_k/U_{k+1}) \\
& = (G_B/V_{k+1})/(V_k/V_{k+1}) \\
& \cong_R G_B/V_k,
\end{align*}
and so $\sigma_{k+1}$ induces an $R$-module isomorphism $\sigma_k$ between $G_A/U_k$ and $G_B/V_k$. It remains to check that $\sigma_{k+1}(U_k/U_{k+1}) = V_k/V_{k+1}$. Since $\sigma_{k+1}$ is an $R$-module isomorphism from $G_A/U_{k+1}$ to $G_B/V_{k+1}$, there is an open $R$-submodule $L$ of $G_A$ containing $U_{k+1}$ such that $\sigma_{k+1}(L/U_{k+1}) = V_k/V_{k+1}$. Thus
\begin{align*}
G_A/L & \cong_R (G_A/U_{k+1})/(L/U_{k+1}) \\
& \cong_R (G_B/V_{k+1})/(V_k/V_{k+1}) \\
& \cong_R G_B/V_k.
\end{align*}
By the argument used in the second claim, we have that $L$ is an $R$-submodule of $U_k$, so that $G_A/U_k$ is an $R$-submodule of $G_A/L\cong_R G_B/V_k$ with $\vert G_A/U_k\vert=\vert G_B/V_k\vert$. Hence $\vert G_A/U_k\vert = \vert G_A/L\vert$, and so
\[ \vert G_A/L\vert= \vert G_A/U_k\vert\,\vert U_k:L\vert\]
implies that $\vert U_k/L\vert =1$, i.e., $L=U_k$. Thus we have that $\sigma(U_k/U_{k+1}) = V_k/V_{k+1}$, giving the fourth claim.

Denote by $\xi_{k+1,k}:Z_{k+1}\to Z_k$ the map defined by $\sigma_{k+1}\to\sigma_k$ as given by the fourth claim. Then $\{ Z_k,\xi_{k+1,k}\}$ is an inverse system of nonempty sets. Hence (by Proposition 1.1.4 in \cite{RZ}), there is an element $(\sigma_k)_{k=1}^\infty$ in the inverse limit
\[ Z =\lim_{\longleftarrow}{}_{k\in{\mathbb N}} Z_k.\]

The sets $\{ G_A/U_k\}_{k=1}^\infty$ and $\{ G_B/V_k\}_{k=1}^\infty$ are naturally surjective inverse systems of finite $R$-modules.  Let $\vartheta_{k+1,k,A}:G_A/U_{k+1}\to G_A/U_k$ and $\vartheta_{k+1,k,B}:G_B/V_{k+1}\to G_B/V_k$ be the canonical epimorphisms, and let
\[\vartheta_{k,A}:\lim_{\longleftarrow} {}_{l\in{\mathbb N}} G_A/U_l\to G_A/U_k, {\rm \ and\ }\vartheta_{k,B}:\lim_{\longleftarrow} {}_{l\in{\mathbb N}} G_B/U_l\to G_B/V_k\]
be the compatible continuous canonical epimorphisms, i.e., $\varphi_{k+1,k,A}\varphi_{k+1,A}=\varphi_{k,A}$ and $\varphi_{k_1,k,B}\varphi_{k+1,B}=\varphi_{k,B}$. The element $(\sigma_k)_{k=1}^\infty\in Z$ is an $R$-module isomorphism between the inverse systems $\{ G_A/U_k\}_{k=1}^\infty$ and $\{ G_B/V_k\}_{k=1}^\infty$, i.e., $\sigma_k\vartheta_{k+1,k,A}= \vartheta_{k+1,k,B}\sigma_{k+1}$ for all $k\in{\mathbb N}$. The continuous epimorphisms $\sigma_k\vartheta_{k,A}$ are compatible, i.e., they satisfy
\[ \vartheta_{k+1,k,B}(\sigma_{k+1}\vartheta_{k+1,A}) = \sigma_k\vartheta_{k+1,k,A}\vartheta_{k+1,A} = \sigma_k \vartheta_{k,A}, \ k\in{\mathbb N},\]
and so induce a topological $R$-module isomorphism
\[ \sigma:\lim_{\longleftarrow} {}_{k\in{\mathbb N}} G_A/U_k \to \lim_{\longleftarrow} {}_{k\in{\mathbb N}} G_B/V_k.\]
The inverse systems $\{ G_A/U_k\}_{k=1}^\infty$ and $\{ G_B/V_k\}_{k=1}^\infty$ are cofinal subsystems of the surjective inverse systems $\{ G_A/U:U\in{\mathcal U}_A\}$ and $\{ G_B/V:V\in{\mathcal U}_B\}$, and so
\[ \lim_{\longleftarrow}{}_{U\in{\mathcal U}_A} G_A/U \cong_R \lim_{\longleftarrow}{}_{k\in{\mathbb N}} G_A/U_k \cong_R  \lim_{\longleftarrow}{}_{k\in{\mathbb N}} G_B/V_k \cong_R \lim_{\longleftarrow}{}_{V\in{\mathcal U}_B} G_B/V,\]
where each of the $R$-module isomorphisms is a topological $R$-module isomorphism.  Here, the first and last profinite $R$-modules are topologically $R$-module isomorphic to $G_A$ and $G_B$ respectively. Therefore, there is a topological $R$-module isomorphism between $G_A$ and $G_B$.
\end{proof}

\section{Characterization of Conjugacy} For similar $A,B\in{\rm GL}(n,{\mathbb Z})$, the profinite $R$-modules $G_A$ and $G_B$ are invariants of conjugacy. Furthermore, they provide means to characterize when $A$ and $B$ are conjugate or not in terms of the embedded copies of the $R$-modules $H_A^\prime$ and $H_B^\prime$ in $G_A$ and $G_B$ respectively. Specifically, for a topological $R$-module isomorphism $\Psi:G_A\to G_B$, it is how the image $\Psi(\iota_A({\mathbb Z}^n))$ intersects $\iota_B({\mathbb Z}^n)$.

\begin{lemma}\label{isomorphism} Suppose $A,B\in{\rm GL}(n,{\mathbb Z}^n)$ are similar and hyperbolic. If there are $R$-module isomorphisms $\Psi_k:G_{k,A}\to G_{k,B}$ such that
\[ \varphi_{k,l,B}\Psi_k = \Psi_l \varphi_{k,l,A}\] 
for all $l\leq k$, then the collection of $R$-module isomorphisms $\Psi_k$ induces a topological $R$-module isomorphism $\Psi:G_A\to G_B$.
\end{lemma}

\begin{proof} The isomorphisms $\Psi_k$ are homeomorphisms because of the discrete topologies on $G_{k,A}$ and on $G_{k,B}$. Assuming that $\varphi_{k,l,B}\Psi_k = \Psi_l \varphi_{k,l,A}$ for all $l\leq k$ implies that the $\Psi_k$ are components of a morphism from the inverse system $\{G_{k,A},\varphi_{k,l,A},{\mathbb N}\}$ to the inverse system $\{ G_{k,B},\varphi_{k,l,B},{\mathbb N}\}$. The maps $\Psi_k \varphi_{k,A}:G_A\to G_{k,B}$ form a collection of compatible maps: for all $l\leq k$, there holds
\begin{align*}
\varphi_{k,l,B} (\Psi_k \varphi_{k,A})\big(\{[m_j]_{j,A}\}_{j\in{\mathbb N}}\big)
& = \varphi_{k,l,B}\Psi_k([m_k]_{k,A}) \\
& = \Psi_l\varphi_{k,l,A}([m_k]_{k,A}) \\
& = \Psi_l([m_k]_{l,A}) \\
& = \Psi_l([m_l]_{l,A}) \\
& = \Psi_l\varphi_{l,A}\big(\{[m_j]_{j,A}\}_{j\in{\mathbb N}}\big).
\end{align*}
The compatible maps $\Psi_k\varphi_{k,A}$ induce a continuous homomorphism $\Psi:G_A\to G_B$ that satisfies $\varphi_{k,B} \Psi = \Psi_k \varphi_{k,A}$ for all $k\in{\mathbb N}$. The map $\Psi$ is given by
\[ \Psi(\{[m_k]_{k,A}\}_{k\in{\mathbb N}}) = \{\Psi_k([m_k]_{k,A})\}_{k\in{\mathbb N}},\]
and is a topological isomorphism because each $\Psi_k$ is a topological isomorphism and because $G_A$ is compact and $G_B$ is Hausdorff.

Since each $\Psi_k$ is an $R$-module isomorphism, we have $\Psi_k  A_k = B_k \Psi_k$ for all $k\in{\mathbb N}$. From this we get
\begin{align*}
\Gamma_B \Psi\big(\{[m_k]_{k,A}\}_{k\in{\mathbb N}}\big)
& = \Gamma_B\big(\Psi_k (\{[m_k]_{k,A}\}_{k\in{\mathbb N}})\big)\\
& = \{B_k \Psi_k  ([m_k]_{k,A})\}_{k\in{\mathbb N}} \\
& = \{\Psi_k A_k  ([m_k]_{k,A})\}_{k\in{\mathbb N}} \\
& = \Psi\big(\{A_k([m_k]_{k,A})\}_{k\in{\mathbb N}}\big) \\
& = \Psi\Gamma_A(\{[m_k]_{k,A}\}_{k\in{\mathbb N}}).
\end{align*}
Thus, $\Psi$ is a topological $R$-module isomorphism.
\end{proof}

\begin{theorem}\label{characterization} For similar hyperbolic $A,B\in{\rm GL}(n,{\mathbb Z})$, the following are equivalent:
\begin{enumerate}
\item[(a)] $A$ is conjugate to $B$,
\item[(b)] there exists a topological $R$-module isomorphism $\Psi:G_A\to G_B$ such that $\Psi\big( \iota_A({\mathbb Z}^n)\big) = \iota_B({\mathbb Z}^n)$, and 
\item[(c)] there exists a topological $R$-module isomorphism $\Psi:G_A\to G_B$ such that  $\Psi\big(\iota_A({\mathbb Z}^n)\big)\cap \iota_B({\mathbb Z}^n)$ is $R$-module isomorphic to the right ${\mathbb Z}[B]$-module ${\mathbb Z}^n$.
\end{enumerate}
\end{theorem} 

\begin{proof} (a)$\Rightarrow$(b) Suppose $A$ and $B$ are conjugate, i.e., there is $C\in{\rm GL}(n,{\mathbb Z})$ such that $AC = CB$. For $k\in{\mathbb N}$ define maps $\Psi_k:G_{k,A}\to G_{k,B}$ by
\[ \Psi_k([m]_{k,A}) = [mC]_{k,B}.\]
We show that the maps $\Psi_k$ are well-defined. For $[m]_{k,A} = [m^\prime]_{k,A}$ we have  $m-m^\prime \in {\mathbb Z}^n(A^{k!}-I)$, and so
\[ mC - m^\prime C = (m-m^\prime)C \in {\mathbb Z}^n(A^{k!}-I)C = {\mathbb Z}^nC(B^{k!}-I) = {\mathbb Z}^n(B^{k!} - I).\]
Hence $\Psi_k([m]_{k,A}) = [mC]_{k,B} = [m^\prime C]_{k,B} = \Psi_k([m^\prime]_{k,A})$.

We show that each $\Psi_k$ is a topological $R$-module isomorphism. For $[m]_{k,A}$ and $ [m^\prime]_{k,A}$ in $G_{k,A}$, we have
\begin{align*}
\Psi_k([m]_{k,A} + [m^\prime]_{k,A})
& = \Psi_k([m+m^\prime]_{k,A}) \\
& = [(m+m^\prime)C]_{k,B} \\
& = [mC + m^\prime C]_{k,B} \\
& = [mC]_{k,B} + [m^\prime C]_{k,B} \\
& = \Psi_k([m]_{k,A}) + \Psi_k([m^\prime]_{k,A}),
\end{align*}
and so each $\Psi_k$ is a homomorphism. Suppose that $\Psi_k([m]_{k,A}) = \Psi_k([m^\prime]_{k,A})$. Then $[mC]_{k,B} = [m^\prime C]_{k,B}$, i.e., $(m-m^\prime)C = mC-m^\prime C\in {\mathbb Z}^n (B^{k!}-I)$. Hence
\[ m-m^\prime \in {\mathbb Z}^n(B^{k!}-I)C^{-1} = {\mathbb Z}^nC^{-1}(B^{k!}-I) = {\mathbb Z}^n(A^{k!}-I).\]
This means that $[m]_{k,A} = [m^\prime]_{k,A}$, and so each $\Psi_k$ is injective. For $[m]_{k,B}\in G_{k,B}$ the element $[m C^{-1}]_{k,A}$ satisfies $\Psi_k([m C^{-1}]_{k,A}) = [m C^{-1}C]_{k,B} = [m]_{k,B}$, and so each $\Psi_k$ is surjective. Since $\Psi_k$ is a continuous bijection, $G_{k,A}$ is compact, and $G_{k,B}$ is Hausdorff, each $\Psi_k$ is a homeomorphism. The maps $\Psi_k$ satisfy
\begin{align*}
\Psi_kA_k ([m]_{k,A})
& = \Psi_k([mA]_{k,A}) \\
& = [mAC]_{k,B} \\
& = [mCB]_{k,B} \\
& = B_k([mC]_{k,B}) \\
& = B_k\Psi_k([m]_{k,A}).
\end{align*}

We show that the $\Phi_k$ induce a topological $R$-module isomorphism from $G_A$ to $G_B$. For $l\leq k$, the topological $R$-module isomorphisms $\Psi_k$ satisfy
\begin{align*}
\varphi_{k,l,B}\Psi_k([m]_{k,A}) 
& = \varphi_{k,l,B}([mC]_{k,B}) \\
& = [mC]_{l,B} \\
& = \Psi_l([m]_{l,A}) \\
& = \Psi_l \varphi_{k,l,A}([m]_{k,A}).
\end{align*} 
By Lemma \ref{isomorphism}, the maps $\Psi_k$ therefore induce a topological $R$-module isomorphism $\Psi:G_A\to G_B$ defined by
\[ \Psi\big(\{[m_k]_{k,A}\}_{k\in{\mathbb N}}\big) = \{\Psi_k([m_k]_{k,A})\}_{k\in{\mathbb N}} = \{[m_kC]_{k,B}\}_{k\in{\mathbb N}}.\]

We show that this topological $R$-module isomorphism $\Psi$ satisfies $\Psi\big(\iota_A({\mathbb Z}^n)\big)=\iota_B({\mathbb Z}^n)$. For $m\in{\mathbb Z}^n$, we have
\[ \Psi\big(\iota_A(m)\big) = \Psi\big(\{[m]_{k,A}\}_{k\in{\mathbb N}}\big) = \{[mC]_{k,B}\}_{k\in{\mathbb N}},\]
which implies that  $\Psi\big(\iota_A({\mathbb Z}^n)\big)\subset \iota_B({\mathbb Z}^n)$. The invertibility of $C$ implies the opposite inclusion, so that $\Psi(\iota_A({\mathbb Z}^n)\big) = \iota_B({\mathbb Z}^n)$.

(b)$\Rightarrow$(c) This follows because $\iota_B$ is an $R$-module monomorphism.

(c)$\Rightarrow$(a). Let $\Delta_\Psi = \Psi\big(\iota_A({\mathbb Z}^n)\big)\cap \iota_B({\mathbb Z}^n)$, and suppose there is a $R$-module isomorphism $h:{\mathbb Z}^n\to \Delta_\Psi$, i.e., $hB=\Gamma_Bh$. Then $h^{-1}:\Delta_\Psi\to {\mathbb Z}^n$ is also a $R$-module isomorphism, i.e., $Bh^{-1}=h^{-1}\Gamma_B$, and so the map $h^{-1}\Psi\iota_A$ is an $R$-module isomorphism from the right ${\mathbb Z}[A]$-module ${\mathbb Z}^n$ to the right ${\mathbb Z}[B]$-module ${\mathbb Z}^n$. This implies that  $h^{-1}\Psi\iota_A$ is an automorphism of ${\mathbb Z}^n$, and so there is $C\in{\rm GL}(n,{\mathbb Z})$ such that $C= h^{-1}\Psi\iota_A$. Thus for all $m\in{\mathbb Z}^n$, we have 
\begin{align*}
mAC &= C(mA) \\
& =  (h^{-1}\Psi\iota_A A)(m) \\
& = (h^{-1}\Psi \Gamma_A \iota_A)(m) \\
& = (h^{-1}\Gamma_B\Psi\iota_A)(m) \\
& = (Bh^{-1}\Psi\iota_A)(m) \\
& = (BC)(m) = mCB.
\end{align*}
Since $mAC=mCB$ holds for all $m\in{\mathbb Z}^n$, we have that $AC=CB$, i.e., that $A$ and $B$ are conjugate.
\end{proof}

We detail another proof of part (b) implying part (a) of Theorem \ref{characterization} that highlights the role that hyperbolicity and the profinite topology play. Suppose that there is a topological $R$-module isomorphism $\Psi:G_A\to G_B$ such that $\Psi\big(\iota_A({\mathbb Z}^n)\big) = \iota_B({\mathbb Z}^n)$. Since $\iota_B:{\mathbb Z}^n\to G_B$ is a monomorphism, there is an isomorphism $\iota_B^{-1}:\iota_B({\mathbb Z}^n)\to{\mathbb Z}^n$, and so $\Psi\big( \iota_A({\mathbb Z}^n) \big)= \iota_B({\mathbb Z}^n)$ implies that
\[ \iota_B^{-1}\Psi \iota_A:{\mathbb Z}^n\to{\mathbb Z}^n\]
is an automorphism. Since the automorphism group of ${\mathbb Z}^n$ is ${\rm GL}(n,{\mathbb Z})$, there is $C\in{\rm GL}(n,{\mathbb Z})$ such that
\[ \Psi \iota_A = \iota_B C.\]
For all $m\in{\mathbb Z}^n$, it follows that
\[ \Psi\big(\{[m]_{k,A}\}_{k\in{\mathbb N}}\big) = (\Psi\iota_A)(m) = \big(\iota_B C\big)(m) = \{[mC]_{k,B}\}_{k\in{\mathbb N}}.\]
Since $\Psi$ is an $R$-module isomorphism, we have $\Gamma_B \Psi = \Psi \Gamma_A$, and so for all $m\in{\mathbb Z}^n$,
\begin{align*}
\{[mCB]_{k,B}\}_{k\in{\mathbb N}} 
& = \Gamma_B(\{[mC]_{k,B}\}_{k\in{\mathbb N}}) \\
& = \Gamma_B\Psi(\{[m]_{k,A}\}_{k\in{\mathbb N}}) \\
& = \Psi\Gamma_A(\{[m]_{k,A}\}_{k\in{\mathbb N}}) \\
& = \Psi(\{[mA]_{k,A}\}_{k\in{\mathbb N}}) \\
& = \{[mAC]_{k,B}\}_{k\in{\mathbb N}}.
\end{align*}
This means for all $k\in{\mathbb N}$ and for all $m\in{\mathbb Z}^n$ that
\[  m(CB-AC)=mCB-mAC  \in N_{k,B} = {\mathbb Z}^n(B^{k!}-I).\]
For the standard basis of row vectors $e_1,\dots,e_n$ of ${\mathbb Z}^n$, we have $e_j(CB-AC) \in N_{k,B}$ for all $k\in{\mathbb N}$ and for all $j=1,\dots,n$. By Lemma \ref{NA}, the hyperbolicity of $B$ implies that
\[ \bigcap_{k\in{\mathbb N}} N_{k,B} = \{0\}.\]
Thus $e_j(CB-AC)= 0$ for $j=1,\dots,n$. Therefore $AC=CB$, and so $A$ and $B$ are conjugate. This completes this alternative proof of (b) implies (a) in Theorem \ref{characterization}.

When similar hyperbolic $A,B\in{\rm GL}(n,{\mathbb Z})$ are nonconjugate but strongly BF-equivalent, it is not the lack of a topological $R$-module isomorphism between $G_A$ and $G_B$ that is responsible for lack of conjugacy between $A$ and $B$. By Theorem \ref{SBFimpliesPsi} there exists at least one topological $R$-module isomorphism $\Psi:G_A\to G_B$. By part (b) of Theorem \ref{characterization}, the failure of $\Psi(\iota_A({\mathbb Z}^n))=\iota_B({\mathbb Z}^n)$ for {\it every} topological $R$-module isomorphism $\Psi:G_A\to G_B$ is responsible for the lack of conjugacy between $A$ and $B$. For instance, there is no way to take {\it any} topological $R$-module isomorphism from $G_A$ to $G_B$ and, by adjusting $G_A$ and $G_B$ by topological $R$-module automorphisms, achieve a topological $R$-module isomorphism $\Psi:G_A\to G_B$ such that $\Psi(\iota_A({\mathbb Z}^n))=\iota_B({\mathbb Z}^n)$. By part (c) of Theorem \ref{characterization}, similar statements hold regarding the failure of $\Psi\big(\iota_A({\mathbb Z}^n)\big)\cap \iota_B({\mathbb Z}^n)$ to be $R$-module isomorphic to the right ${\mathbb Z}[B]$-module ${\mathbb Z^n}$.

The $R$-submodule $\Delta_\Psi$ of $\iota_B({\mathbb Z} ^n)$ defined in the proof of part (c) of Theorem \ref{characterization} plays the key role in the characterization of conjugacy for strongly BF-equivalent similar hyperbolic $A,B\in{\rm GL}(n,{\mathbb Z})$. Since $\Psi(\iota_A({\mathbb Z}^n))=\Psi(\iota_A(\theta_A^{-1}(H_A^\prime)))$ and $\iota_B({\mathbb Z}^n)=\iota_B(\theta_B^{-1}(H_B^\prime))$ as $R$-modules, then
\[ \Delta_\Psi = \Psi(\iota_A(\theta_A^{-1}(H_A^\prime)))\cap\iota_B(\theta_B^{-1}(H_B^\prime)) .\]
The characterizations of conjugacy given Theorem \ref{characterization} imply that strongly BF-equivalent similar hyperbolic $A,B\in{\rm  GL}(n,{\mathbb Z})$ are conjugate if and only if there exists a topological $R$-module isomorphism $\Psi:G_A\to G_B$ such that either
\[ \Psi(\iota_A(\theta_A^{-1}(H_A^\prime)))=\iota_B(\theta_B^{-1}(H_B^\prime)),\]
or  $\Psi(\iota_A(\theta_A^{-1}(H_A^\prime)))\cap\iota_B(\theta_B^{-1}(H_B^\prime))\cong_R H_B^\prime$. Thus the embedded copies of $H_A^\prime$ and $H_B^\prime$ in $G_A$ and $G_B$ respectively, determine the conjugacy of $A$ with $B$ or the lack thereof.

\section{Irreducible Hyperbolic Toral Automorphisms} No assumption is made about the irreducibility of the characteristic polynomial of the similar hyperbolic ${\rm GL}(n,{\mathbb Z})$ matrices in Theorem \ref{SBFimpliesPsi} or Theorem \ref{characterization}. A hyperbolic toral automorphism $A\in{\rm GL}(n,{\mathbb Z})$ is irreducible if its characteristic polynomial is irreducible. For a topological $R$-module isomorphism $\Psi:G_A\to G_B$, more can be deduced about the nature of the intersection of $\Psi(\iota_A({\mathbb Z}^n)$ with $\iota_B({\mathbb Z}^n)$ when $A$ and $B$ have the same irreducible characteristic polynomial.

\begin{theorem}\label{finiteintersection} Suppose similar hyperbolic $A,B\in{\rm GL}(n,{\mathbb Z})$ are irreducible. Then for any topological $R$-module isomorphism $\Psi:G_A\to G_B$, either $\Psi\big(\iota_A({\mathbb Z}^n)\big)\cap \iota_B({\mathbb Z}^n)=\{0\}$ or $\Psi\big(\iota_A({\mathbb Z}^n)\big)\cap\iota_B({\mathbb Z}^n) $ has finite index in $\iota_B({\mathbb Z}^n)$.
\end{theorem}

\begin{proof} Suppose $\Psi:G_A\to G_B$ is a topological $R$-module isomorphism such that
\[ \Delta_\Psi = \Psi\big(\iota_A({\mathbb Z}^n)\big)\cap \iota_B({\mathbb Z}^n)\ne\{0\}.\]
Since $\iota_A({\mathbb Z}^n)$ is an $R$-submodule of $G_A$ and $\Psi$ is an $R$-module isomorphism, it follows that
\[ \Psi\big(\iota_A({\mathbb Z}^n)\big) =\Psi\big(\Gamma_A\big(\iota_A({\mathbb Z}^n)\big)\big) = \Gamma_B\big(\Psi\big( \iota_A({\mathbb Z}^n)\big)\big),\]
and so $\Psi\big(\iota_A({\mathbb Z}^n)\big)$ is an $R$-submodule of $G_B$. Since $\iota_B({\mathbb Z}^n)$ is an $R$-submodule of $G_B$, it follows that $\Delta_\Psi$ is an $R$-submodule of $G_B$. 

The inclusion of $R$-submodules $\Delta_\Psi< \iota_B({\mathbb Z}^n)$ and the conjugacy $\Gamma_B\iota_B = \iota_BB$ imply that $\iota_B^{-1}(\Delta_\Psi)$ is a $B$-invariant subgroup of ${\mathbb Z}^n$:
\[ B(\iota_B^{-1}(\Delta_\Psi)) =\iota_B^{-1}(\Gamma_B(\Delta_\Psi)) = \iota_B^{-1}(\Delta_\Psi).\]
Let $S$ be the nontrivial subspace ${\mathbb R}^n$ spanned by a set of generators of $\iota^{-1}_B(\Delta_\Psi)$. This subspace is $B$-invariant, i.e., $B(S)=S$, and it projects to a right action $T_B$-invariant subtorus of ${\mathbb T}^n$. The irreducibility of the characteristic polynomial of $B$ implies that the only $T_B$-invariant subtori of ${\mathbb T}^n$ are $\{0\}$ and ${\mathbb T}^n$ (see Proposition 3.1 on p.\,726 in \cite{KKS}). Since $S$ is nontrivial, it follows that $\iota_B^{-1}(\Delta_\Psi)$ contains a basis for ${\mathbb R}^n$. Thus $\iota_B^{-1}(\Delta_\Psi)$ has finite index in ${\mathbb Z}^n$, and hence $\Delta_\Psi$ has finite index in $\iota_B({\mathbb Z}^n)$.
\end{proof}

When $\Delta_\Psi\ne\{0\}$, then by Theorem \ref{finiteintersection}, $\Delta_\Psi$ has finite index within $\iota_B({\mathbb Z}^n)$ and so $\Delta_\Psi$ is group-theoretically isomorphic to ${\mathbb Z}^n$. By part (c) of Theorem \ref{characterization}, only when $\Delta_\Psi$ is $R$-module isomorphic to the right ${\mathbb Z}[B]$-module ${\mathbb Z}^n$ can we conclude that $A$ and $B$ are conjugate. For any irreducible similar hyperbolic strongly BF-equivalent $A$ and $B$ that are not conjugate, such as in Example \ref{ex2}, any topological $R$-module isomorphism $\Psi:G_A\to G_B$ has either $\Delta_\Psi=\{0\}$ or it has finite index in $\iota_B({\mathbb Z}^n)$, but there is no topological $R$-module isomorphism $\Psi:G_A\to G_B$ with $\Delta_\Psi= \Psi(\iota_A(\theta_A^{-1}(H_A^\prime)))\cap\iota_B(\theta_B^{-1}(H_B^\prime))$ being $R$-module isomorphic $H_B^\prime$, and thus the embedded copies of $H_A^\prime$ and $H_B^\prime$ in $G_A$ and $G_B$ respectively, distinguish $A$ and $B$ dynamically.


\begin{thebibliography}{99}

\bibitem{AP} R. Adler  and R. Palais, {\it Homeomorphic Conjugacy of Automorphisms on the Torus}, Proc. Amer. Math. Soc., Vol. 16, No. 6 (1965) 1222-1225.

\bibitem{ATW} R. Alder, C. Tresser, and P.A. Worfolk, {\it Topological Conjugacy of Linear Endomorphisms of the $2$-Torus}, Trans. Amer. Math. Soc. 349 No. 4 (1997), 1633-1652.

\bibitem{Ba} L.F. Bakker, {\it Rigidity of Projective Conjugacy of Quasiperiodic Flows of Koch Type}, Colloq. Math. Vol. 112, No. 2 (2008), 291-312.

\bibitem{Co} H. Cohen, ``A Course in Computational Algebraic Number Theory,'' Graduate Texts in Mathematics Vol. 138, Springer-Verlag, 1993.

\bibitem{DTZ} E.C. Dade, O. Taussky, H.  Zassenhaus, {\it On the Theory of Orders, in particular on the semigroup of ideal classes and genera of an order in an algebraic number field}, Math. Annalen Vol. 148 (1962), 31-64.

\bibitem{GS} F. Gruneward and D. Segal, {\it Some General Algorithms. I: Arithmetic Groups}, Ann. Math. Vol. 112, No. 3 (1980), 531-583.

\bibitem{KH} A. Katok, and B. Hasselblatt, ``Introduction to the Modern Theory of Dynamical Systems'', Encyclopedia of Mathematics and its Applications, Vol. 54, Cambridge University Press, 1995.

\bibitem{KKS} A. Katok, S. Katok, and K. Schmidt, {\it Rigidity of measurable structures for ${\mathbb Z}^d$-actions by automorphims of a torus}, Comment. Math. Helv. 77 (2002), 718-745.

\bibitem{Ko} G. Kolutsky, {\it Geometric Continued Fractions as Invariants in the Topological Classification of Anosov Diffeomorphisms of Tori}, Iteration Theory (ECIT'08), 2009, Bericht Nr. 354, pp. 99-111.

\bibitem{LS} D. Lind and K. Schmidt., {\it Homoclinic Points of Algebraic $\mathbb Z^d$-Actions}, J. Amer. Math. Soc., Vol. 12, No. 4 (1999), 953-980.

\bibitem{MS1} P. Martins Rodrigues and J. Sousa Ramos, {\it Algebraic Results on Markov Shifts of Torus Automorphisms}, in L. Magalh\~aes et al., International Conference on Differential Equations (Equadiff 95) 436-441, World Scientific, Singapore, 1998.

\bibitem{MS1.5}  P. Martins Rodrigues and J. Sousa Ramos, {\it Symbolic Representations and ${\mathbb T}^n$ Automorphisms}, Int. J. of Bifurcation and Chaos, Vol. 13, No. 7 (2003), 2005-2010.

\bibitem{MS2} P. Martins Rodrigues, and J. Sousa Ramos, {\it Bowen-Franks groups as conjugacy invariants for ${\mathbb T}^n$-automorphisms}, Aequationes Math. 69 (2005), 231-249.

\bibitem{Ne} M. Newman, ``Integral Matrices,'' Pure and Applied Mathematics, Vol. 45, Academics Press, 1972.

\bibitem{RZ} L. Ribes and P. Zalesskii, ``Profinite Groups,'' A Series of Modern Surveys in Mathematics, Vol. 40, Springer-Verlag, 2000.

\bibitem{SD} H.P.F. Swinnerton-Dyer, ``A Brief Guide to Algebraic Number Theory,'' London Mathematical Society, Cambridge, 2001.

\bibitem{Ta} O. Taussky ,{\it Connections between algebraic number theory and integral matrices,} in H. Cohn, ``A Classical Invitation to Algebraic Numbers and Class Fields,'' Springer-Verlag, 1978.

\bibitem{Wa} D.I. Wallace, {\it Conjugacy Classes of Hyperbolic Matrices in ${\rm SL}(n,{\mathbb Z})$ and Ideal Classes in an Order}, Trans. Amer. Math. Soc., Vol. 283, No. 1 (1984), 177-184.

\end{thebibliography}
\end{document}